\documentclass[a4paper]{amsart}

\usepackage{graphicx}
\usepackage[english]{babel}

\usepackage{amsfonts,amssymb,amsmath}
\theoremstyle{theorem}
\newtheorem{theorem}{Theorem}
\newtheorem{lemma}[theorem]{Lemma}
\newtheorem{proposition}[theorem]{Proposition}
\newtheorem{corollary}[theorem]{Corollary}
\theoremstyle{definition}
\newtheorem{definition}[theorem]{Definition}
\theoremstyle{remark}
\newtheorem{remark}[theorem]{Remark}

\providecommand{\dd}{\mathrm{d}}
\DeclareMathOperator{\dom}{dom}
\DeclareMathOperator{\lspan}{span}

\begin{document}
\title[Efficient numerics for HJM models]{Efficient simulation and calibration of general HJM models by splitting schemes}
\author{Philipp D\"orsek \and Josef Teichmann}
\maketitle
\begin{abstract}
We introduce efficient numerical methods for generic HJM equations of interest rate theory by means of high-order weak approximation schemes.
These schemes allow for QMC implementations due to the relatively low dimensional integration space. The complexity of the resulting algorithm is considerably lower than the complexity of multi-level MC algorithms as long as the optimal order of QMC-convergence is guaranteed. In order to make the methods applicable to real world problems, we introduce and use the setting of weighted function spaces, such that unbounded payoffs and unbounded characteristics of the equations in question are still allowed. We also provide an implementation, where we efficiently calibrate an HJM equation to caplet data.
\end{abstract}
\section{Introduction}


The Heath-Jarrow-Morton equation (HJM-equation) of interest rate theory (\cite{HeathJarrowMorton1992}; see \cite{Filipovic2001,CarmonaTehranchi2006,Filipovic2009} for expositions) is a stochastic partial differential equation (SPDE) on the state space of forward rate curves, which is flexible enough to describe complicated dynamical features such as non-constant (local or stochastic) volatility, non-constant correlation, or jumps, or dependence structures. 
An analysis of geometric properties was performed in \cite{FilipovicTeichmann2004}.
As forward rate curves already encode all the market's information on default-free bond prices, it only remains to estimate volatilities either from the time series or from option prices or from both of them. 
For this purpose it is required that the numerical treatment of the HJM-equation can be performed efficiently: it is the purpose of this article to actually show that efficient numerical methods for the HJM-equation are at hand, how to construct and how to implement them.

In the case of generic SPDEs we usually neither have sufficient analytical information on the marginal's distribution, nor on its Fourier-Laplace transform, nor its short-time asymptotics. 
We are therefore forced to apply simulation techniques to approximate the random variables in question and we face two main sources of problems in such a procedure:

\subsection{Discretization error}
Numerical weak or strong approximation schemes with probabilistic flavor are built upon stochastic Taylor expansion and its iteration along $n$ steps due to the Markov property. 
Depending on the local error of the method this leads (at least for some class of test functions) to a global error $ \mathcal{O}(1/n^s) $, which is called error of order $ s $. 
The method is called high order method if $ s >1 $, and standard or low order method otherwise. 
There are schemes, e.g.~cubature methods \cite{Kusuoka2001,Kusuoka2004,LyonsVictoir2004} or splitting methods \cite{NinomiyaVictoir2008,NinomiyaNinomiya2009}, that substantially increase $ s $ and therefore reduce the global discretization error for a fixed number of discretization steps $n$. 
When applying the theory of weighted spaces we can also enlarge the sets of test functions and generic equations, which can be treated by the discretization method.

\subsection{Integration error}
Having discretized the SPDE problem we still have to evaluate the random variables involved in each local step, which usually leads to a numerical integration problem on some $\mathbb{R}^{dn}$, where $ d $ is the fixed number of dimensions which are needed for each local discretization step. 
Here we can apply three approaches: (deterministic) numerical integration, Monte-Carlo algorithms (MC) or Quasi-Monte-Carlo algorithms (QMC). 
Due to the $n$-dependence of integration space we do not try a direct numerical integration method, even though we have some hope that such an approach could possibly work. 
MC algorithms lead to integration errors $ \mathcal{O}(1/\sqrt{K}) $, where $K$ denotes the number of integration points, whereas QMC algorithms lead to integration errors approximately $ \mathcal{O}(1/K) $. 
In both cases the integration error dominates the total error asymptotically, which can be seen by complexity analysis. 
Let us be more precise on this: we assume that $ \mathcal{O}(n) $ operations are performed to calculate the value of the functional which we intend to integrate. Here we tacitly assume that dealing with elements in state space is $ \mathcal{O}(1) $, which is strictly speaking only guaranteed to be true in a finite dimensional setting. However dealing, e.g., with curves on the real line numerically can still be of $ \mathcal{O}(1) $ if only the relevant parts of the curve are actually calculated. Hence the total complexity $ C $ of the method is $ d_3 n K $, where $ d_3 $ is a constant. Given an accuracy $ \epsilon $ the following inequality has to hold true additionally,
\[
\frac{d_1}{n^s} + \frac{d_2}{\sqrt{K}} \leq \epsilon,
\]
whence we end up with the simple constraint minimization problem to minimize the complexity \[ C=d_3 n K \to \min \] given the previous inequality on accuracy. 
Its asymptotic solution is given by $ C=\mathcal{O}(\epsilon^{-2-1/s}) $ with $ K = \mathcal{O}(\epsilon^{-2})$ and $ n = \mathcal{O}(\epsilon^{-1/s})$. 
This can be improved by multi-level methods \cite{Heinrich2001,Giles2008,CreutzigDereichMuellerGronbachRitter2009} to a complexity estimate of order almost $\mathcal{O}(\epsilon^{-2})$, which is in turn the complexity of one dimensional MC integration. 
In other words, the complexity is equal to the integration of a functional where evaluating at a single point is of order $ \mathcal{O}(1) $. 
Multi-level methods improve by telescoping errors on different levels of discretization $ n $. 
However, in this case the asymptotic complexity is not improved by higher-order methods anymore, since it depends on weak and strong convergence orders so that one is restricted to low order Euler-like methods. 

\bigskip

If  we perform the same complexity analysis in case of higher order discretization schemes with a QMC algorithm instead of an MC algorithm we obtain an asymptotic complexity $ C=\mathcal{O}(\epsilon^{-1-1/s}) $ with $ K= \mathcal{O}(\epsilon^{-1}) $ and $ n = \mathcal{O}(\epsilon^{-1/s})$, which is indeed considerably better than multi-level MC in case of higher order methods ($s>1$). 
On the other hand it is not better than multi-level QMC \cite{GilesWaterhouse2009} theoretically could be. 
We emphasize that a multi-level QMC is theoretically far from being understood, and additionally we would need strong order $ 1 $ methods which are not always at hand. An additional problematic aspect is the need of high-dimensional integration spaces, where QMC is not known to perform well anymore. 

We claim that standard QMC with high-order weak approximation schemes is superior to multilevel MC due to the low dimensionality of the integration space $ \mathbb{R}^{dn} $ \emph{as long as the QMC order of convergence is understood to hold true}. 
For accuracy $ \epsilon $ the dimension of integration space is of order $ \mathcal{O}(\epsilon^{-1/s}) $, which in real world implementations is often sufficiently small such that the QMC order of convergence is ensured.

The goal of this work is therefore twofold: first we want to show how actually the theory of weighted spaces applies to the HJM equation. 
We even show that we have a simple weak approximation method of order $ 2 $ within this setting. 
Second, we claim that a QMC algorithm integrating the resulting functional is numerically efficient. 
We underline this statement by a calibration of a time-homogeneous, non-linear, diffusive HJM-equation to caplet data, i.e., we calibrate this equation to ten volatility smiles simultaneously.
Our method is not only fast enough for the calibration of the model, but also the computer programming itself is almost as easy as a standard Euler-Maruyama scheme due to the use of a splitting approach.

Let us compare our results to well-known and recent results on splitting schemes and weak approximation methods for SPDEs.
In contrast to classical results on the Lie-Trotter splitting such as \cite{BensoussanGlowinski1989,BensoussanGlowinskiRascanu1990,FlorchingerLeGland1991,Bensoussan1992,BensoussanGlowinskiRascanu1992,RascanuTudor1995,LeGland1992,SunGlowinski1993,ItoRozovskii2000,Gyongy2002,GyongyKrylov2003a,GyongyKrylov2003b}, we focus on a higher order method for nonlinear problems in the spirit of \cite{NinomiyaVictoir2008}, hence allowing us to conclude the practical efficiency of the method as explained above.
The topic of weak approximation for SPDEs was recently analysed in \cite{Debussche2011}.
While there, the focus was on space-time white noise driving the system, we consider only finite-dimensional noise, and can obtain the same rate of convergence as in the finite-dimensional setting with bounded and smooth vector fields.
Contrary to \cite{Ninomiya2010}, our model is inherently infinite-dimensional and does not allow a reduction to a low-dimensional stochastic differential equation.

\section{Weighted spaces and analysis of stochastic partial differential equations}
\label{sec:weightedspaces}
We provide an overview of the theory of weighted spaces that is at the core of the presented numerical method.
For more details, see also \cite{RoecknerSobol2006,DoersekTeichmann2010,Doersek2011,Doersek2011phdthesis}.
\subsection{The generalised Feller condition}
Given a fixed $\ell\ge1$, we consider the following setup.
\begin{enumerate}
	\item For $i=0,\cdots,\ell$, $(H_i,\langle\cdot,\cdot\rangle_{H_i})$ is a separable Hilbert space, and its norm is denoted by $\lVert x\rVert_{H_i}:=\langle x,x\rangle_{H_i}^{1/2}$.
	\item $H_{i+1}$ is compactly and densely embedded into $H_i$ for $i=0,\dots,\ell-1$.
	\item $A\colon\dom A\subset H_0\to H_0$ is the generator of a strongly continuous semigroup of contractions $(S_t)_{t\ge 0}$ on $H_0$.
	\item For $i=0,\dots,\ell-1$, $A\colon H_{i+1}\to H_i$ is bounded.
	\item For $i=1,\dots,\ell$, $(S_t)_{t\ge0}$ is a strongly continuous semigroup of contractions on $H_i$.
\end{enumerate}
In many cases, it will be adequate to choose $H_i:=\dom A^{i}$, e.g., if $A$ is a differential operator on a bounded domain.
If, however, $A$ is a differential operator on an unbounded domain, $\dom A$ will usually not be compactly embedded into $H_0$.
As we are interested in the HJM equation, where the underlying space variable varies in $[0,\infty)$, we consider the above, more general setup.
\begin{definition}
	Let $i\in\{0,\dots,\ell\}$.
	Given a left-continuous, increasing function $\rho\colon [0,\infty)\to(0,\infty)$ with $\lim_{u\to\infty}\rho(u)=+\infty$, set $\psi_i(x):=\rho(\lVert x\rVert_{H_i})$, we define the \emph{enveloping space} $\mathrm{B}^{\psi_i}_k(H_i):=\left\{ f\in\mathrm{C}^k(H_i)\colon \lVert f\rVert_{\psi_i,k}<\infty \right\}$, where $\mathrm{C}^k(X)$ denotes the space of $k$ times continuously Fr\'echet differentiable functions and
	\begin{equation}
		\lVert f\rVert_{\psi_i,k}
		:=
		\sum_{i=0}^{k}\lvert f\rvert_{\psi_i,j}
		\quad\text{with}\quad
		\lvert f\rvert_{\psi,j}
		:=
		\sup_{x\in H_i}\psi_i(x)^{-1}\lVert D^j f(x)\rVert_{L_j(H_i)}.
	\end{equation}
	Here, $L_j(H_i)$ is the linear space of bounded multilinear forms $a\colon (H_i)^j\to \mathbb{R}$ endowed with the norm 
	\begin{equation}
		\lVert a\rVert_{L_j(H_i)}
		:=
		\sup_{\substack{x_1,\dots,x_j\in H_i \\ \lVert x_i\rVert_{H_i}, i=1,\dots,j}}\lvert a(x_1,\dots,x_j)\rvert,
	\end{equation}
	which makes $(L_j(H_i),\lVert\cdot\rVert_{L_j(H_i)})$ a Banach space.

	Given an orthonormal basis $(e_j)_{j\in\mathbb{N}}$, define the space $\mathcal{A}(H_i)$ of bounded smooth cylindrical functions $H_i\to\mathbb{R}$ by
	\begin{equation}
		\mathcal{A}(H_i)
		:=
		\left\{ f\colon \text{$f=g(\langle\cdot,e_1\rangle_{H_i},\cdots,\langle\cdot,e_N\rangle_{H_i})$ for some $N\in\mathbb{N}$ and $g\in\mathrm{C}_b^{\infty}(\mathbb{R}^{N})$} \right\}.
	\end{equation}
	The closure of $\mathcal{A}(H_i)$ in $\mathrm{B}^{\psi_i}_k(H_i)$ is denoted by $\mathcal{B}^{\psi_i}_k(H_i)$, $k\ge 0$.
\end{definition}
\begin{remark}
	\label{rem:generalweightfunc}
	The above assumptions on the weight function $\psi_i$ are very restrictive.
	A weaker assumption on the weight function on which our analysis can be performed would be that the sets $\left\{ x\in H_i\colon \psi_i(x)\le R \right\}$ are weakly compact, and hence bounded, in $H_i$, and that $\psi_i$ is bounded on bounded sets.
	This is applied in Section~\ref{subsec:HJM-mma}.
\end{remark}
Applying \cite[Corollary~5.3, Remark~5.4]{RoecknerSobol2006}, we see that our space $\mathcal{B}^{\psi_i}_0(H_i)$ coincides with the space $WC_{\psi_i}$ defined by M.~R\"ockner and Z.~Sobol.
Hence, the following result is proved in \cite[Theorem~5.1]{RoecknerSobol2006}.
\begin{proposition}
	\label{prop:dualspaceBpsi}
	There exists an isometric isomorphy from $\mathcal{B}^{\psi}_0(H_i)^{*}$, the dual space to $\mathcal{B}^{\psi}_0(H_i)$, to the space 
	\begin{equation}
		\mathcal{M}^{\psi}(H_i)
		:=
		\left\{ \mu\colon\text{$\mu$ is a signed Borel measure on $H_i$ with $\int_{H_i}\psi_i(x)\lvert\mu(\dd x)\rvert<\infty$} \right\},
	\end{equation}
	where the latter space is endowed with the norm $\lVert\mu\rVert_{\psi_i,*}:=\int_{H_i}\psi_i(x)\lvert\mu(\dd x)\rvert$, $\lvert\mu\rvert$ denoting the total variation measure to $\mu$.
	The inverse of this isometry is given by $\ell_\mu(f):=\int_{H_i} f(x)\mu(\dd x)$ for all $f\in\mathcal{B}^{\psi_i}_0(H_i)$ and $\mu\in\mathcal{M}^{\psi}(H_i)$.
\end{proposition}
This result allows us to obtain a generalisation of the well-known Feller condition for the strong continuity of operator semigroups on $\mathrm{C}_0(D)$, $D$ a locally compact topological space, to the infinite-dimensional setting.
\begin{corollary}[generalised Feller condition]
	\label{cor:genfellercond}
	Fix $i\in{0,\dots,\ell}$.
	Let $(P_t)_{t\ge 0}$ be a family of continuous operators on $\mathcal{B}^{\psi_i}_0(H_i)$ satisfying the \emph{generalised Feller condition}, i.e.,
	\begin{enumerate}
		\item $P_0=I$, the identity on $\mathcal{B}^{\psi_i}_0(H_i)$,
		\item $P_{t+s}=P_t P_s$ for $s$, $t\ge 0$,
		\item $\lVert P_t\rVert_{L(\mathcal{B}^{\psi_i}_0(H_i))}\le C$ for all $t\in[0,\varepsilon)$ with some $C>0$ and $\varepsilon>0$, where $L(\mathcal{B}^{\psi_i}_0(H_i))$ is the space of bounded and linear operators on $\mathcal{B}^{\psi_i}_0(H_i)$ and endowed with the operator norm, and finally
		\item $\lim_{t\to 0+}P_t f(x)=f(x)$ for all $x\in X$ and $f\in\mathcal{B}^{\psi_i}_0(H_i)$.
	\end{enumerate}
	Then, $(P_t)_{t\ge 0}$ is a strongly continuous semigroup on $\mathcal{B}^{\psi_i}_0(H_i)$, i.e., for every $f\in\mathcal{B}^{\psi_i}_0(H_i)$, $\lim_{t\to 0+}\lVert P_t f-f\rVert_{\psi_i}=0$.
\end{corollary}
\begin{proof}
	This is an easy consequence of Proposition~\ref{prop:dualspaceBpsi} and \cite[Theorem~I.5.8]{EngelNagel2000}:
	We only need to prove that $\lim_{t\to 0+}\ell(P_t f)=\ell(f)$ for all $f\in\mathcal{B}^{\psi_i}_0(H_i)$ and $\ell\in\mathcal{B}^{\psi_i}_0(H_i)^*$.
	But by Proposition~\ref{prop:dualspaceBpsi}, $\ell(f)=\int_{X}f(x)\mu(\dd x)$ for all $f\in\mathcal{B}^{\psi_i}_0(H_i)$ with some $\mu\in\mathcal{M}^{\psi_i}(H_i)$.
	As $\lim_{t\to 0+}P_t f(x)=f(x)$ for all $x\in H_i$, an application of Lebesgue's dominated convergence theorem yields the claim.
\end{proof}
Hence, in contrast to the weak continuity of Markov semigroups for infinite dimensional stochastic equations \cite{Cerrai1994}, the above result allows us to work with standard strongly continuous semigroups.

Usually, the difficult part in verifying the generalised Feller condition for a given Markov semigroup $(P_t)_{t\ge 0}$ is proving that $P_t(\mathcal{B}^{\psi_i}_0(H_i))\subset\mathcal{B}^{\psi_i}_0(H_i)$.
The following result can often be applied to this problem.
\begin{theorem}
	\label{thm:CbkBpsikdense}
	For $k\ge0$ and $i=1,\dots,\ell$, $\mathrm{C}_b^k(H_{i-1})\subset\mathcal{B}^{\psi}_k(H_i)$ is dense.
\end{theorem}
\begin{proof}
	Apply Proposition~\ref{prop:simultaneousOB} to obtain an orthonormal basis $(e_n)_{n\in\mathbb{N}}$ of $H_{i-1}$ that is simultaneously orthogonal in $H_i$.
	Defining $\mathcal{A}(H_i)$ using $(e_n/\lVert e_n\rVert_{H_i})_{n\in\mathbb{N}}$, we see that every $f=g(\langle\cdot,e_1/\lVert e_1\rVert_{H_i}\rangle_{H_i},\dots,\langle\cdot,e_N/\lVert e_N\rVert_{H_i}\rangle_{H_i})\in\mathcal{A}(H_i)$ can be extended to a smooth cylindrical function on $H_{i-1}$, as
	\begin{alignat}{2}
		& g(\langle\cdot,e_1/\lVert e_1\rVert_{H_i}\rangle_{H_{i}},\dots,\langle\cdot,e_N/\lVert e_N\rVert_{H_i}\rangle_{H_i}) \\
		&=
		g(\lVert e_1\rVert_{H_i}\langle\cdot,e_1\rangle_{H_{i-1}},\dots,\lVert e_N\rVert_{H_i}\langle\cdot,e_N\rangle_{H_{i-1}}).\nonumber
	\end{alignat}
	Whence $\mathcal{A}(H_i)\subset\mathrm{C}_b^k(H_{i-1})$.

	Next, we show $\mathrm{C}_b^k(H_{i-1})\subset\mathcal{B}^{\psi}_k(H_i)$.
	Given $f\in\mathrm{C}_b^k(H_{i-1})$ and $\varepsilon>0$, we shall construct $f_{\varepsilon}\in\mathcal{A}(H_i)$ such that $\lVert f-f_{\varepsilon}\rVert_{\psi_i,k}<\varepsilon$.
	Let $\pi_N$ denote the $H_{i-1}$-orthogonal projection onto $\lspan\{e_j\colon j=1,\dots,N\}$.
	For $R>0$ arbitrary, we estimate
	\begin{alignat}{2}
		\lVert f-f\circ \pi_N\rVert_{\psi_i,k}
		&\le\sum_{j=1}^{k}\sup_{\substack{x\in H_i\\\lVert x\rVert_{H_i}\le R}}\psi_i(x)^{-1}\lVert D^j f(x) - D^j f(\pi_N x)\rVert_{L_j(H_i)}
		\notag\\
		&\phantom{\le}+ \rho(R)^{-1}\lVert f-f\circ\pi_N\rVert_{\mathrm{C}_b^k(H_i)}
	\end{alignat}
	As $\pi_N\colon H_i\to H_{i-1}$ is of operator norm one for all $N\in\mathbb{N}$, it is easy to see by the properties of $\rho$ that the final term goes to zero as $R$ goes to infinity, and hence can be made smaller than $\varepsilon/3$ by choosing $R_{\varepsilon}$, depending on $f$ but not on $N$, large enough.
	For the first term, note that $B:=\{x\in H_i\colon\lVert x\rVert_{H_i}\le R_{\varepsilon}\}$ is precompact in $H_{i-1}$.
	Hence, there exists $\delta>0$ such that $\lVert D^j f(x)-D^j f(y)\rVert_{L_j(H_i)}<\varepsilon/3$ for $x$, $y\in B$ with $\lVert x-y\rVert_{H_{i-1}}<\delta$.
	Choose $N_{\delta}$ according to Corollary~\ref{cor:simultaneousOBapproximation} to obtain that $\lVert x-\pi_N x\rVert_{H_{i-1}}<\delta$ whenever $N\ge N_{\delta}$ and $x\in B$.

	Finally, choose $f_{\varepsilon}\colon H_{i,N_{\delta}}\to\mathbb{R}$ in such a way that $\sum_{j=0}^{k}\sup_{\substack{x\in H_{i,N_{\delta}}\\\lVert x\rVert_{H_i}}}\lVert D^j f_{\varepsilon}(x)-D^j(f\circ\pi_{N_\delta})(x)\rVert_{L_j(H_i)}<\varepsilon/3$ and $\lVert f_{\varepsilon}\rVert_{\mathrm{C}_b^k(H_{i,N_{\delta}})}\le\lVert f\circ\pi_{N_\delta}\rVert_{\mathrm{C}_b^k(H_{i,N_{\delta}})}$.
	Here, $H_{i,N_{\delta}}:=\lspan\left\{ e_j\colon j=1,\dots,N_{\delta} \right\}$.
	Such a choice is always possible, as $H_{i,N_{\delta}}$ is finite dimensional and we can thus apply a standard mollifying argument.
	It follows similarly as above that $\lVert f_{\varepsilon}-f\circ\pi_{N_{\delta}}\rVert_{\psi_i,k}<\varepsilon/3$, and plugging the results together, we obtain
	\begin{equation}
		\lVert f-f_{\varepsilon}\rVert_{\psi_i,k} < \varepsilon.
	\end{equation}
	Thus, $\mathcal{A}(H_i)\subset\mathrm{C}_b^k(H_{i-1})\subset\mathcal{B}^{\psi_i}_k(H_i)$, and the claim follows.
\end{proof}
\begin{theorem}
	\label{thm:strongcont}
	Fix $i\in\left\{ 1,\dots,\ell \right\}$.
	Let $(x(t,x_0))_{t\ge 0}$ be a time homogeneous Markov property on the stochastic basis $(\Omega,\mathcal{F},(\mathcal{F}_t)_{t\ge 0},\mathbb{P})$ with values in $H_{i-1}$.
	Assume that
	\begin{enumerate}
		\item the mapping $H_{i-1}\to H_{i-1}$, $x_0\mapsto x(t,x_0)$ is almost surely continuous with respect to the norm topology on $H_i$ for every $t\ge 0$,
		\item if $x_0\in H_i$ and $t\ge 0$, then $x(t,x_0)\in H_i$ almost surely,
		\item for some $\varepsilon>0$ and $C>0$, $\mathbb{E}[\psi_i(x(t,x_0))]\le C\psi(x_0)$ for all $x_0\in H_i$ and $t\in[0,\varepsilon]$, and 
		\item $(x(t,x_0))_{t\ge0}$ has almost surely c\`adl\`ag paths in the weak topology of $H_i$.
	\end{enumerate}
	Then, $P_t\in L(\mathcal{B}^{\psi}_0(X))$ for all $t\ge 0$, where $P_t f(x_0):=\mathbb{E}[f(x(t,x_0))]$, $(P_t)_{t\ge 0}$ satisfies the generalised Feller condition, and hence, $(P_t)_{t\ge 0}$ is a strongly continuous semigroup on $\mathcal{B}^{\psi}_0(X)$.
\end{theorem}
\begin{proof}
	First, we prove that $\lim_{t\to0+}P_t f(x_0)=f(x_0)$ for fixed $f\in\mathcal{B}^{\psi_i}_0(H_i)$ and $x_0\in H_i$.
	Let $R>\lVert x_0\rVert_{H_i}$.
	Set $B_R:=\left\{ x\in H_i\colon \lVert x\rVert_{H_i}\le R \right\}$, then
	\begin{alignat}{2}{}
		\lvert P_t f(x_0) - f(x_0)\rvert
		&\le
		\mathbb{E}[\lvert f(x(t,x_0)) - f(x_0)\rvert] \notag \\
		&\le
		\mathbb{E}[\lvert f(x(t,x_0)) - f(x_0)\rvert\chi_{B_R}(x(t,x_0))] \notag\\
		&\phantom{\le}+ \mathbb{E}[\lvert f(x(t,x_0))\rvert\chi_{H_i\setminus B_R}(x(t,x_0))] \notag\\
		&\phantom{\le}+ \lvert f(x_0)\rvert\mathbb{P}[\lVert x(t,x_0)\rVert_{H_i}> R]
	\end{alignat}
	with $\chi_A(x):=1$, $x\in A$, $0$ otherwise the indicator function of the set $A$.
	The Markov inequality yields
	\begin{equation}
		\mathbb{P}[\lVert x(t,x_0)\rVert_{H_i}>R]
		\le\rho(R)^{-1}\mathbb{E}[\psi_i(x(t,x_0))],
	\end{equation}
	and this term goes to zero as $R$ goes to infinity.
	Furthermore,
	\begin{equation}
		\mathbb{E}[\lvert f(x(t,x_0))\rvert\chi_{H_i\setminus B_R}(x(t,x_0))]
		\le
		\lVert f\rVert_{\psi_i,0}\mathbb{E}[\psi_i(x(t,x_0))\chi_{H_i\setminus B_R}(x(t,x_0))],
	\end{equation}
	and dominated convergence proves that this also goes to zero as $R$ goes to infinity.
	Finally, note that $f|_{B_R}$ is weakly continuous, as $f\in\mathcal{B}^{\psi_i}_0(H_i)$.
	Weak compactness of $B_R$ yields that $\lvert f(x)-f(x_0)\rvert\le2\sup_{x\in B_R}\lvert f(x)\rvert<\infty$ for $x\in B_R$, and monotone convergence proves $\lim_{t\to0+}\mathbb{E}[\lvert f(x(t,x_0))-f(x_0)\rvert\chi_{B_R}(x(t,x_0))]=0$.
	Hence, we have shown $\lim_{t\to0+}P_t f(x_0)=f(x_0)$.

	Next, note that $P_t(\mathrm{C}_b(H_{i-1}))\subset\mathrm{C}_b(H_{i-1})$, which is a consequence of the assumption of almost sure continuity of the mapping $x_0\mapsto x(t,x_0)$.
	As 
	\begin{alignat}{2}{}
		\lVert P_t f\rVert_{\psi_i,0}
		&=
		\sup_{x_0\in H_i}\psi_i(x_0)^{-1}\lvert \mathbb{E}[f(x(t,x_0))]\rvert \notag\\
		&\le
		\lVert f\rVert_{\psi_i,0}\sup_{x\in H_i}\psi_i(x_0)^{-1}\mathbb{E}[\psi_i(x(t,x_0))]
		\le C\lVert f\rVert_{\psi_i,0}
	\end{alignat}
	for $t\in[0,\varepsilon]$, this proves by Theorem~\ref{thm:CbkBpsikdense} that $P_t\in L(\mathcal{B}^{\psi_i}_0(H_i))$ for $t\in[0,\varepsilon]$.
	As the semigroup property is satisfied due to the Markov property, an induction shows $P_t\in L(\mathcal{B}^{\psi_i}_0(H_i))$ for all $t\ge 0$.
	Thus, Corollary~\ref{cor:genfellercond} proves the claim.
\end{proof}

\subsection{Application to stochastic partial differential equations}
\label{subsec:weightedspaces-application}
Let $(x(t,x_0))_{t\ge 0}$ be the solution of the stochastic partial differential equation
\begin{subequations}
\label{eq:spde}
\begin{alignat}{2}
	\dd x(t,x_0) &= ( Ax(t,x_0) + V_0(x(t,x_0)) )\dd t + \sum_{j=1}^{d}V_j(x(t,x_0))\circ\dd W^j_t, \\
	x(0,x_0)     &= x_0.
\end{alignat}
\end{subequations}
Here, $(W_t)_{t\ge 0}$ is a $d$-dimensional standard Brownian motion.
The vector fields $V_j$ are assumed to be of the form $V_j(x)=g_j(Lx)$, where $g_j\in\mathrm{C}_b^{\infty}(\mathbb{R}^N;H_{\ell})$ is a smooth function on $\mathbb{R}^N$ with values in $H_{\ell}$, and $L\colon H_0\to\mathbb{R}^N$ is a bounded linear mapping.
These are typical assumptions for HJM models to be applied in practice, see \cite{FilipovicTeichmann2004}.
Then, it follows that \eqref{eq:spde} admits unique solutions in every space $H_i$, $i=0,\dots,\ell$, given that the initial value $x_0$ is smooth enough.
\begin{lemma}
	Fix $\beta>0$ and $i\in\left\{ 0,\dots,\ell \right\}$.
	For some $\varepsilon>0$, there exists $C>0$ such that
	\begin{equation}
		\mathbb{E}[\cosh(\beta\lVert x(t,x_0)\rVert_{H_i})]\le C\cosh(\beta\lVert x_0\rVert_{H_i})
		\quad\text{for $x_0\in H_i$ and $t\in[0,\varepsilon]$}.
	\end{equation}
\end{lemma}
\begin{proof}
	We apply It\^o's formula.
	For $m\ge2$,
	\begin{alignat}{2}{}
		\dd\lVert x(t,x_0)\rVert_{H_i}^{2m}
		&=
		m\lVert x(t,x_0)\rVert_{H_i}^{2(m-1)}\langle x(t,x_0),\dd x(t,x_0)\rangle_{H_i} 
		\notag \\ &\phantom{=}
		+ \frac{1}{2}m\Bigl( (m-1)\lVert x(t,x_0)\rVert_{H_i}^{2(m-2)}\langle x(t,x_0),\dd x(t,x_0)\rangle_{H_i}^{2} 
		\notag \\ &\phantom{=+\frac{1}{2}m\Bigl(}
		+ \lVert x(t,x_0)\rVert_{H_i}^{2(m-1)}\langle \dd x(t,x_0),\dd x(t,x_0)\rangle_{H_i} \Bigr)
		\notag \\
		&=
		m\lVert x(t,x_0)\rVert_{H_i}^{2(m-1)}\Bigl( \langle x(t,x_0),Ax(t,x_0)\rangle_{H_i}\dd t 
		\notag \\ &\phantom{=m\lVert x(t,x_0)\rVert_{H_i}^{2(m-1)}}
		+ \langle x(t,x_0),V_0(x(t,x_0))\rangle_{H_i}\dd t
		\notag \\ &\phantom{=m\lVert x(t,x_0)\rVert_{H_i}^{2(m-1)}}
		+ \sum_{j=1}^{d}\langle x(t,x_0),V_j(x(t,x_0))\rangle_{H_i}\dd W^j_t \Bigr)
		\notag \\
		&\phantom{=}+ \frac{1}{2}m\lVert x(t,x_0)\rVert_{H_i}^{2(m-2)}\sum_{j=1}^{d}\Bigl( (m-1)\langle x(t,x_0),V_j(x(t,x_0))\rangle_{H_i}^2 
		\notag \\ &\phantom{=}
		\qquad\qquad\qquad+ \lVert x(t,x_0)\rVert_{H_i}^2\langle V_j(x(t,x_0)),V_j(x(t,x_0))\rangle_{H_i} \Bigr)\dd t.
	\end{alignat}
	Taking expectations, the boundedness of the $V_j$ and the dissipativity of $A$ yield, as all moments are uniformly bounded by \cite[Theorem~7.3.5]{DaPratoZabczyk2002}, a constant $C>0$ independent of $m\ge2$ such that
	\begin{equation}
		\mathbb{E}[\lVert x(t,x_0)\rVert_{H_i}^{2m}]
		\le
		\lVert x_0\rVert_{H_i}^{2m} + Cm\int_{0}^{t}\mathbb{E}\left[ \lVert x(s,x_0)\rVert_{H_i}^{2m-1} + m\lVert x(s,x_0)\rVert_{H_i}^{2(m-1)} \right]\dd s.
	\end{equation}
	For $m=1$, we similarly obtain
	\begin{equation}
		\mathbb{E}[\lVert x(t,x_0)\rVert_{H_i}^2]
		\le
		\lVert x_0\rVert_{H_i}^2 + C\int_{0}^{t}\mathbb{E}[\lVert x(s,x_0)\rVert_{H_i} + 1]\dd s,
	\end{equation}
	and trivially, $\mathbb{E}[\lVert x(t,x_0)\rVert^0] = 1$.
	Note that $\cosh(u)=\sum_{m=0}^{\infty}\frac{u^{2m}}{(2m)!}$.
	Summing up, the monotone convergence theorem proves
	\begin{alignat}{2}{}
		\mathbb{E}[\cosh(\beta\lVert x(t,x_0)\rVert_{H_i})]
		&\le
		\cosh(\beta\lVert x_0\rVert_{H_i})
		\notag \\ &\phantom{\le}
		+ C\beta\int_{0}^{t}\mathbb{E}\Biggl[ \sum_{m=1}^{\infty}\frac{m}{(2m)!}\beta^{2m-1}\lVert x(s,x_0)\rVert_{H_i}^{2m-1}
		\notag \\ &\phantom{\le+C\beta}
		+ \beta\sum_{m=1}^{\infty}\frac{m^2}{(2m)!}\beta^{2(m-1)}\lVert x(s,x_0)\rVert_{H_i}^{2(m-1)} \Biggr]
		\notag \\
		&\le
		\cosh(\beta\lVert x_0\rVert_{H_i})
		\notag \\ &\phantom{\le}
		+C\frac{\beta}{2}\int_{0}^{t}\mathbb{E}\bigl[\sinh(\beta\lVert x(s,x_0)\rVert_{H_i})
		\notag \\ &\phantom{\le+C\beta}
		+\beta\cosh(\beta\lVert x(s,x_0)\rVert_{H_i})\bigr]\dd s.
	\end{alignat}
	Here, we have used that $\sinh(u)=\sum_{m=1}^{\infty}\frac{u^{2m-1}}{(2m-1)!}$, and that $\frac{m}{2m-1}\le 1$ for $m\ge 1$.
	As $\sinh(u)\le\cosh(u)$, we obtain that with a constant $C>0$ depending on $\beta$,
	\begin{alignat}{2}{}
		\mathbb{E}[\cosh(\beta\lVert x(t,x_0)\rVert_{H_i})]
		&\le
		\cosh(\beta\lVert x_0\rVert_{H_i}) 
		\notag \\ &\phantom{\le}
		+ C\int_{0}^{t}\mathbb{E}[\cosh(\beta\lVert x(s,x_0)\rVert_{H_i})]\dd s.
	\end{alignat}
	The method of the moving frame (see \cite{Teichmann2009}) allows us to conclude that \newline $\mathbb{E}[\cosh(\beta\lVert x(t,x_0)\rVert_{H_i})]<\infty$ for $t\ge 0$.
	Hence, an application of Gronwall's inequality proves the claim.
\end{proof}
Hence, the choice of weight function $\psi_{i,\beta}(x):=\cosh(\beta\lVert x\rVert_{H_i})$, $\beta>0$, is appropriate.
This is particularly important in the application of our results to the HJM equation, see Section~\ref{sec:HJM}.
\begin{corollary}
	Given $i\in\left\{ 1,\dots,\ell \right\}$ and $\beta>0$, the Markov semigroup $(P_t)_{t\ge 0}$ of $(x(t,x_0))_{t\ge 0}$ is strongly continuous on $\mathcal{B}^{\psi_{i,\beta}}_0(H_i)$.
\end{corollary}
\begin{proof}
	Under the given assumptions, we can prove the conditions of Theorem~\ref{thm:strongcont} using \cite[Theorem~7.3.5]{DaPratoZabczyk2002}.
\end{proof}
Choose some $\ell_0\in\left\{ 1,\dots,\ell \right\}$ and $\beta_0>0$.
We perform an analysis of the infinitesimal generator $\mathcal{G}$ with domain $\dom{G}$ of $(P_t)_{t\ge 0}$, considered as strongly continuous semigroup on $\mathcal{B}^{\psi_{\ell_0,\beta_0}}(H_{\ell_0})$.
In the following, $Vf(x):=Df(x)(V(x))$ denotes the directional derivative for sufficiently smooth functions $f\colon H_i\to\mathbb{R}$ and vector fields $V\colon H_i\to H_i$.
\begin{lemma}
	\label{lem:estimatesmoothlieder}
	Fix $i\in\left\{ 0,\dots,\ell \right\}$.
	For $j=0,\dots,d$ and $f\in\mathcal{A}(H_i)$, $V_j f\in\mathcal{A}(H_i)$.
	Furthermore, the directional derivative $f\mapsto V_j f$ defines a bounded linear operator from $\mathcal{B}^{\psi_{i,\beta}}_k(H_i)$ to $\mathcal{B}^{\psi_{i,\beta}}_{k-1}(H_i)$, $k\ge 1$.
\end{lemma}
\begin{proof}
	The special form of $V_j$ proves $V_j f\in\mathcal{A}(H_i)$ for $f\in\mathcal{A}(H_i)$.
	The estimate $\lVert V_j f\rVert_{\psi_{i,\beta}}\le C\lVert f\rVert_{\psi_{i,\beta}}$ can be shown by a direct calculation using the boundedness of $V_j$ and its derivatives, and the result follows from the density of $\mathcal{A}(H_i)$ in $\mathcal{B}^{\psi_{i,\beta}}_k(H_i)$.
\end{proof}
\begin{lemma}
	\label{lem:estimategeneratorlieder}
	Fix $i\in\left\{ 1,\dots,\ell \right\}$ and $\beta_1<\beta_2$.
	The operator $f\mapsto Df(\cdot)(A\cdot)$ maps $\mathcal{A}(H_{i-1})$ to $\mathcal{A}(H_i)$, and defines a bounded linear operator from $\mathcal{B}^{\psi_{i-1,\beta_1}}_k(H_{i-1})$ to $\mathcal{B}^{\psi_{i,\beta_2}}_{k-1}(H_i)$, $k\ge 1$.
\end{lemma}
\begin{proof}
	Given $f\in\mathcal{A}(H_{i-1})$, there exists a $H_{i-1}$-orthogonal projection $\pi$ with finite-dimensional range such that $f\circ\pi=f$.
	Hence, $Df(x)(Ax)=Df(x)(\pi Ax)$, and it is easy to see that this function is in $\mathcal{A}(H_i)$.
	The boundedness is again shown by a direct calculation, where we apply that $u\cosh(\beta_1 u)\le C\cosh(\beta_2 u)$ for all $u\in[0,\infty)$ with some constant $C>0$.
\end{proof}
An application of It\^o's formula, see \cite[Theorem~7.2.1]{DaPratoZabczyk2002}, yields
that for $i\in\left\{ 1,\dots,\ell \right\}$ and $f\in\mathcal{A}(H_{i-1})$, 
\begin{equation}
	\label{eq:infgen}
	\mathcal{G}f(x)
	=
	Df(x)(Ax) + (V_0f)(x) + \frac{1}{2}\sum_{j=1}^{d}(V_j^{2}f)(x)
	\quad\text{for $x\in H_i$}.
\end{equation}
\begin{theorem}
	\label{thm:mappinginfgen}
	Fix $i\in\left\{ 1,\dots,\ell \right\}$.
	For $j\ge0$ and $0<\beta_1<\beta_2$, the operator $\tilde{\mathcal{G}}\colon\mathcal{B}^{\psi_{\beta_1}}_{j+2}(H_{i-1})\to\mathcal{B}^{\psi_{\beta_2}}_j(H_{i})$, given by the right hand side of \eqref{eq:infgen}, is well-defined as a bounded linear operator.
	Furthermore, for $\beta\in(0,\beta_0)$, $\mathcal{B}^{\psi_{\ell_0-1,\beta}}_2(H_{\ell_0-1})\subset\dom\mathcal{G}$, and on this space, $\mathcal{G}=\tilde{\mathcal{G}}$.
\end{theorem}
\begin{proof}
	The boundedness of $\tilde{\mathcal{G}}$ follows from Lemmas~\ref{lem:estimatesmoothlieder} and \ref{lem:estimategeneratorlieder}.
	For the second property, note that $\tilde{\mathcal{G}}$ maps $\mathcal{B}^{\psi_{\ell_0-1,\beta}}_2(H_{\ell_0-1})$ into $\mathcal{B}^{\psi_{\ell_0,\beta_0}}_0(H_{\ell_0})$ as a bounded linear operator, $\mathcal{G}=\tilde{\mathcal{G}}$ on $\mathcal{A}(H_{\ell_0-1})$, and that $\mathcal{G}$ is a closed operator.
	Hence, $\mathcal{G}=\tilde{\mathcal{G}}$ follows from a density argument.
\end{proof}
\begin{corollary}
	Fix $\beta\in(0,\beta_0)$.
	Given $k\in\left\{ 0,\dots,\ell_0-1 \right\}$, we have the Taylor expansion
	\begin{equation}
		P_t f
		=
		\sum_{j=0}^{k}\frac{t^j}{j!}\mathcal{G}^j f + R_{t,k}f
		\quad\text{for $f\in\mathcal{B}^{\psi_{\ell_0-(k+1),\beta}}_{2(k+1)}(H_{\ell_0-(k+1)})$},
	\end{equation}
	where the operator $R_{t,k}\colon\mathcal{B}^{\psi_{\ell_0-(k+1),\beta}}_{2(k+1)}(H_{\ell_0-(k+1)})\to\mathcal{B}^{\psi_{\ell_0,\beta_0}}(H_{\ell_0})$ is bounded uniformly in $t\in[0,\varepsilon]$ for given $\varepsilon>0$.
\end{corollary}
\begin{proof}
	Theorem~\ref{thm:mappinginfgen} proves that $\mathcal{G}^{j}\colon\mathcal{B}^{\psi_{\ell_0-(k+1),\beta}}_{2(k+1)}(H_{\ell_0-(k+1)})\to\mathcal{B}^{\psi_{\ell_0,\beta_0}}_0(H_{\ell_0})$ is a bounded linear operator for $j=0,\dots,k+1$.
	Hence, a standard Taylor expansion argument can be applied to prove the stated theorem.
\end{proof}
\begin{lemma}
	\label{lem:smoothnesspreserving}
	For $\beta>0$, $k\ge0$ and $i\in\left\{ 1,\dots,\ell \right\}$, $P_t\colon \mathcal{B}^{\psi_{i,\beta}}_k(H_i)\to\mathcal{B}^{\psi_{i,\beta}}_k(H_i)$ is a bounded linear operator.
	Its operator norm is bounded uniformly for $t\in[0,T]$, where $T>0$ can be chosen arbitrarily.
\end{lemma}
\begin{proof}
	This is consequence of smooth dependence on the initial value in $H_{i-1}$. 
	By considering the sensitivity equations, see \cite[Theorem~7.3.6]{DaPratoZabczyk2002}, all derivatives $D_{x_0}^j x(t,x_0)(h_1,\dots,h_j)$ are shown to satisfy bounds of the type
	\begin{equation}
		\mathbb{E}[\lVert D_{x_0}^j x(t,x_0)(h_1,\dots,h_j)\rVert_{H_i}^p]
		\le C_p\left(\lVert h_1\rVert_{H_i}\dotsm\lVert h_j\rVert_{H_i}\right)^p
		\quad\text{for $p\ge 2$},
	\end{equation}
	where $C_p$ is independent of $x_0$.
	The boundedness of $P_t$ in the norms given above then follows from the Cauchy-Schwarz inequality together with the property \newline $c\cosh(2u)\le\cosh(u)^2\le C\cosh(2u)$ for some constants $c$, $C>0$.
	Due to Theorem~\ref{thm:CbkBpsikdense}, it follows that $P_t(\mathcal{A}(H_{i-1}))\subset\mathcal{B}^{\psi_{i,\beta}}_k(H_i)$.
	A density argument proves the claim.
\end{proof}
\begin{corollary}
	For $k\ge 2$, $i\in\left\{ 0,\dots,\ell_0-1 \right\}$ and $\beta\in(0,\beta_0)$, $\mathcal{B}^{\psi_{i,\beta}}_k(H_{i})$ is a core for $\mathcal{G}$.
\end{corollary}
\begin{proof}
	Applying \cite[Proposition~II.1.7]{EngelNagel2000}, this is clear from Lemma~\ref{lem:smoothnesspreserving}, as $\mathcal{B}^{\psi_{i,\beta}}_2(H_{i})\subset\dom\mathcal{G}$ is invariant with respect to the semigroup and dense in $\mathcal{B}^{\psi_{i,\beta}}_0(H_{i})$.
\end{proof}

\section{The rate of convergence of splitting schemes for stochastic partial differential equations}
As numerical discretisation scheme, we suggest the use of a splitting scheme.
Decomposing the drift coefficient further, $V_0=\sum_{m=1}^{M}V_{0,m}$, we define the split problems
\begin{subequations}
	\begin{alignat}{5}{}
		\frac{\dd}{\dd t}x_{0,0}(t,x_0) &= Ax_{0,0}(t,x_0), \\
		\frac{\dd}{\dd t}x_{0,m}(t,x_0) &= V_{0,m}(x_{0,m}(t,x_0)), & m=1,\dots,M, \\
		\dd x_j(t,x_0)                  &= V_j(x_j(t,x_0))\circ\dd W^j_t, & j=1,\dots,d.
	\end{alignat}
\end{subequations}
We stress that all of these problems can be solved by finding the corresponding deterministic flows; in the case of $j=1,\dots,d$; we need to evaluate the flow induced by the vector field $V_j$ at the stochastic time $W^j_t$.
In particular, the processes $(x_{0,m}(t,x_0))_{t\ge 0}$, $m=0,\dots,M$, are deterministic.
The split semigroups are defined by $P^{0,m}_tf(x_0):=f(x_{0,m}(t,x_0))$ and $P^j_t f(x_0):=\mathbb{E}[f(x_j(t,x_0))]$, $j=1,\dots,d$.
We consider the following splitting schemes.
\begin{description}
	\item[Lie-Trotter splitting, forward ordering]
		The Lie-Trotter splitting with forward ordering is of first order and reads
		\begin{equation}
			Q^{\mathrm{LTfwd}}_{(\Delta t)}f
			:=
			P^{0,0}_{\Delta t}P^{0,1}_{\Delta t}\dots P^{0,M}_{\Delta t}P^1_{\Delta t}\dots P^d_{\Delta t}f
			\quad\text{for $f\in\mathcal{B}^{\psi_{\ell_0,\beta_0}}(H_{\ell_0})$}.
		\end{equation}
	\item[Lie-Trotter splitting, backward ordering]
		The Lie-Trotter splitting with backward ordering is obtained by reversing the order of the operators in the Lie-Trotter splitting with forward ordering,
		\begin{equation}
			Q^{\mathrm{LTbwd}}_{(\Delta t)}f
			:=
			P^d_{\Delta t}\dots P^1_{\Delta t} P^{0,M}_{\Delta t}\dots P^{0,1}_{\Delta t} P^{0,0}_{\Delta t}f
			\quad\text{for $f\in\mathcal{B}^{\psi_{\ell_0,\beta_0}}(H_{\ell_0})$},
		\end{equation}
		and is also of first order.
	\item[Ninomiya-Victoir splitting]
		The Ninomiya-Victoir splitting is a generalisation of the well-known Strang splitting to more than two generators and reads
		\begin{alignat}{2}{}
			Q^{\mathrm{NV}}_{(\Delta t)}f
			:=
			\frac{1}{2}P^{0,0}_{\Delta t/2}
			\Bigl( 
			&P^{0,1}_{\Delta t}\dots P^{0,M}_{\Delta t}P^1_{\Delta t}\dots P^d_{\Delta t} \notag \\
			&+ P^d_{\Delta t}\dots P^1_{\Delta t}P^{0,M}_{\Delta t}\dots P^{0,1}_{\Delta t} 
			\Bigr)
			P^{0,0}_{\Delta t/2} f
			\quad\text{for $f\in\mathcal{B}^{\psi_{\ell_0,\beta_0}}(H_{\ell_0})$}.
		\end{alignat}
		It is of second order.
\end{description}
The theory of Section~\ref{sec:weightedspaces} now applies not only to the continuous semigroup $(P_t)_{t\ge 0}$, but also to every split semigroup $(P^{0,m}_t)_{t\ge 0}$ and $(P^j_t)_{t\ge0}$, yielding spaces invariant to the dynamics of $(P_t)_{t\ge0}$ on which we can apply the generators $\mathcal{G}$, $\mathcal{G}_{0,m}$ and $\mathcal{G}_j$, $m=0,\dots,M$ and $j=1,\dots,d$, and
\begin{equation}
	\mathcal{G}
	=
	\sum_{m=0}^{M}\mathcal{G}_{0,m} + \sum_{j=1}^{d}\mathcal{G}_j.
\end{equation}
Hence, we obtain the following result.
\begin{theorem}
	\label{thm:convrate}
	Let $\beta\in(0,\beta_0)$, and assume that $(Q_{(\Delta t)})_{\Delta t\ge 0}$ is any splitting approximation of $(P_t)_{t\ge 0}$ based on the split semigroups $(P^{0,m}_t)_{t\ge0}$ and $(P^j_t)_{t\ge0}$, $m=0,\dots,M$, $j=1,\dots,d$, which is of formal order $s\in\left\{ 1,\dots,\ell_0-1 \right\}$.
	For $f\in\mathcal{B}^{\psi_{0,\beta}}_{2(s+1)}(H_0)$, 
	\begin{equation}
		\lVert P_t f - Q_{(t/n)}^n f\rVert_{\psi_{s+1,\beta_0}}
		\le C_T n^{-s}\lVert f\rVert_{\psi_{0,\beta},2(s+1)}.
	\end{equation}
\end{theorem}
\begin{proof}
	The theory of \cite{HansenOstermann2009a} yields this result in the following manner.
	Clearly, all split semigroups are stable on the space $\mathcal{B}^{\psi_{s+1,\beta_0}}_0(H_{s+1})$ in the sense that the operator norms of the operators are bounded by $\exp(Ct)$ with some constant $C>0$ independent of $t$ and of the semigroup.
	Furthermore, for every $t\ge 0$, $P_t$ is a bounded linear operator on $\mathcal{B}^{\psi_{0,\beta}}_{2(s+1)}(H_0)$ by Lemma~\ref{lem:smoothnesspreserving}, and on this space, we have that all generators of the split semigroups and the original semigroup are well-defined together with their products, and satisfy
	\begin{equation}
		\left( \sum_{m=0}^{M}\mathcal{G}_{0,m} + \sum_{j=1}^{d}\mathcal{G}_j \right)^\alpha
		=
		\mathcal{G}^\alpha,
		\quad \alpha=0,\dots,s+1.
	\end{equation}
	Hence, we obtain the claimed result from \cite[Theorem~2.3, Sections~4.1, 4.4]{HansenOstermann2009a}.
\end{proof}
\section{Symmetrically weighted sequential splitting}\label{sec:swss}
	Applying the theory of \cite{GyongyKrylov2006,GyongyKrylov2011} allows us to obtain asymptotic expansions for the forward and backward ordering of the Lie-Trotter splitting and the Ninomiya-Victoir splitting if the function $f$ is sufficiently smooth.
	Using symmetry, we can even prove that the Ninomiya-Victoir splitting and the \emph{symmetrically weighted sequential splitting}, going back at least to \cite[equation (25)]{Strang1963} and given by
	\begin{equation}
		Q^{\mathrm{SWSS}}_{t,n}f
		:=
		\frac{1}{2}\left( (Q^{\mathrm{LTfwd}}_{(t/n)})^nf + (Q^{\mathrm{LTbwd}}_{(t/n)})^nf \right),
	\end{equation}
	have asymptotic expansions not only in $n^{-1}$, but even $n^{-2}$.
	Hence, every extrapolation step would improve convergence by two orders.
	In particular, the symmetrically weighted sequential splitting is of second order. Comparing the dimension of integration space of different second         order schemes and in view of possible extrapolations  we use SWSS in our numerical computations detailed below. Indeed dimension of integration space for the Ninomiya-Victoir scheme is $ n (d+1) $, whereas sequential splitting leads to dimension $ n d + 1 $.

\section{Application: the Heath-Jarrow-Morton equation}
\label{sec:HJM}
As application of our theoretical results, we provide a numerical method for the efficient simulation of the Heath-Jarrow-Morton equation of interest rate theory.
It is of the form specified in \eqref{eq:spde}, where the infinitesimal generator is given by the differential operator $\frac{\dd}{\dd x}$.
In order to include a stochastic volatility process, the Hilbert spaces $H_i$, $i=0,\dots,\ell$ are specified as follows.
We set
\begin{alignat}{2}{}
	H_i:=\bigl\{ h\in\mathrm{L}^1_{\mathrm{loc}}( (0,\infty) )\colon
	&\text{$h$ is $i+1$ times weakly differentiable and} \notag \\
	&\text{$h'$,\dots,$h^{(i+1)}\in\mathrm{L}^2_{\alpha_i}( (0,\infty) )$} \bigr\} \times \mathbb{R}.
\end{alignat}
Here, $0<\alpha_0<\dots<\alpha_\ell$, and 
\begin{equation}
\mathrm{L}^2_{\alpha}( (0,\infty) ):=\left\{ h\in\mathrm{L}^1_{\mathrm{loc}}( (0,\infty) )\colon \int_{(0,\infty)}r(x)^2\exp(\alpha x)\dd x<\infty \right\}.
\end{equation}
It is easy to see that $H_{i+1}\subset H_i$ for $i\in\left\{ 0,\dots,\ell-1 \right\}$, and that every function in $H_0$ is continuous and bounded (see also \cite{Filipovic2001}).
The scalar product on $H_i$ reads
\begin{alignat}{2}{}
	\langle (h_1,v_1),(h_2,v_2)\rangle_{H_i}
	&:=
	h_1(0)h_2(0) \notag \\
	&\phantom{:=}+ \sum_{m=1}^{i}\int_{(0,\infty)}h_1^{(m)}(x)h_{2}^{(m)}(x)\exp(\alpha_i x)\dd x + v_1 v_2.
\end{alignat}
With the induced norm, 
\begin{equation}
	A\colon H_{i+1}\to H_i, \quad (h,v)\mapsto(h',-\alpha v),
\end{equation}
where $\alpha\ge 0$ is a constant, becomes a bounded linear operator.
It agrees with the generator of the shift semigroup on the first component of $H_i$, $i\in\left\{ 0,\dots,\ell-1 \right\}$.

Consider the Heath-Jarrow-Morton equation with stochastic volatility in It\^o form,
\begin{subequations}
	\begin{alignat}{2}{}
		\dd r(t,r_0,v_0)
		&=
		(Ar(t,r_0,v_0)+\alpha_{\mathrm{HJM}}(r(t,r_0,v_0),v(t,v_0)))\dd t \notag \\
		&\phantom{=}+ \sum_{j=1}^{d}\sigma_j(r(t,r_0,v_0),v(t,v_0))\dd W^j_t, \\
		\dd v(t,v_0)
		&=
		-\alpha v(t,v_0)\dd t + \sum_{j=1}^{d}\gamma_j\dd W^j_t, \\
		r(0,r_0,v_0) &= r_0, \\
		v(0,v_0) &= v_0.
	\end{alignat}
\end{subequations}
The stochastic volatility $v(t,v_0)$ is chosen as a mean-reverting Ornstein-Uhlenbeck process.
The HJM drift satisfies the condition
\begin{equation}
	\alpha_{\mathrm{HJM}}(h,v)(x)
	=
	\sum_{j=1}^{d}\sigma_j(h,v)(x)\int_{0}^{x}\sigma_j(h,v)(\xi)\dd\xi.
\end{equation}
We assume that $\sigma_j$ are of the form required in Section~\ref{subsec:weightedspaces-application}, i.e., $\sigma_j(h,v)=g_j(Lh,v)$, where $g_j\in\mathrm{C}_b^{\infty}(\mathbb{R}^{N+1};H_{\ell})$ and $L\colon H_0\to\mathbb{R}^N$ is bounded linear.
Rewriting the equation in Stratonovich form, we see that
\begin{equation}
	V_0(h,v)
	=
	\alpha_{\mathrm{HJM}}(h,v) - \frac{1}{2}\sum_{j=1}^{d}D\sigma_j(h,v)(\sigma_j(h,v)),
\end{equation}
and it follows easily that $V_0(h,v)=g_0(Lh,v)$ with some $g_0\in\mathrm{C}_b^{\infty}(\mathbb{R}^{N+1};H_{\ell})$.
Hence, Theorem~\ref{thm:convrate} applies to prove the optimal rate of convergence of $s$ of a splitting scheme for sufficiently smooth functions $f\colon H_0\to\mathbb{R}$, given that the initial value satisfies $(r_0,v_0)\in H_{s+1}$.

\subsection{The money market account}
\label{subsec:HJM-mma}
In order to calculate standard payoffs, we not only need the instantaneous forward curve, but also the money market account $(B_t)_{t\ge 0}$.
It is given by $B_t=\exp(z(t,0))$, where
\begin{equation}
	\dd z(t,r_0,v_0,z_0) = R_t\dd t, \quad z(0,r_0,v_0,z_0) = z_0,
\end{equation}
and can therefore be easily included into our splitting scheme.
Here, we denote by $R_t := r(t,r_0,v_0)(0)$ the short rate induced by our HJM model.

To recover the optimal rate of convergence, we argue as follows.
On the product space $\tilde{H}_i:=H_{i}\times\mathbb{R}$, we consider the weight function $\tilde{\psi}_{i,\beta}(h,v,z):=\psi_{i,\beta}(h,v) + z^2$ (see Remark~\ref{rem:generalweightfunc}).
As proved before,
\begin{equation}
	\mathbb{E}[\psi_{i,\beta}(r(t,r_0,v_0),v(t,v_0))]
	\le
	\exp(Ct)\psi_{i,\beta}(r_0,v_0).
\end{equation}
Furthermore, as $R_t\le\psi_{i,\beta}(r(t,r_0,v_0))$,
\begin{alignat}{2}{}
	\mathbb{E}[z(t,r_0,v_0,z_0)^2]
	&\le z_0^2 + \int_{0}^{t}\mathbb{E}[z(t,r_0,v_0,z_0)^2]\dd s + \int_{0}^{t}\mathbb{E}[R_t^2]\dd s \notag \\
	&\le z_0^2 + \int_{0}^{t}\mathbb{E}[\tilde{\psi}_{i,\beta}(r(t,r_0,v_0),v(t,v_0),z(t,r_0,v_0,z_0))]\dd s.
\end{alignat}
Altogether, an application of Gronwall's inequality proves
\begin{equation}
	\mathbb{E}[\tilde{\psi}_{i,\beta}(r(t,r_0,v_0),v(t,v_0),z(t,r_0,v_0,z_0))]
	\le
	\exp(Ct)\tilde{\psi}_{i,\beta}(r_0,v_0,z_0),
\end{equation}
and we can apply the above theorems to all functions contained in $\mathcal{B}^{\tilde{\psi}_{i,\beta}}_{k}(H_i\times\mathbb{R})$ by evident modifications of the above proofs.

Now, the money market account is not included in the above setting.
More precisely, $B_t=\exp(z(t,r_0,v_0,0))$, and this growth is larger that the quadratic growth admitted by $\tilde{\psi}_{i,\beta}$.
We deal with this problem in the following way:
Actually, $z(t,r_0,v_0,0)$ should be \emph{nonnegative} from an economic point of view.
Hence, we replace the money market account by $\tilde{B}_t:=\exp(\Phi(z(t,r_0,v_0,0)))$, where $\Phi\colon\mathbb{R}\to\mathbb{R}$ is $\mathrm{C}^{\infty}$ with bounded derivatives, satisfies $\Phi(z)=z$ for all $z\ge -K$, and is bounded from below by $-2K$ with some $K>0$.
In our numerical experiments, performed using the model calibrated to the data from \cite{Kluge2005}, we never encountered paths with $z(t,0)\le 0$.
Furthermore, even if $z(t,0)$ becomes slightly negative on some paths, this is numerically innocent, as we can adjust $K$ accordingly.
We want to stress that our modification only acts on economically dubious paths where the money market account falls significantly in the long run, and neither limits temporary decrease, or any increase whatsoever.

Clearly, $\tilde{B}_t^{-1}\le\exp(2K)$.
Hence, the modified payoff of a zero coupon bond with time to maturity $\delta$, 
\begin{equation}
	f(h,v,z)
	:=
	\exp(-\Phi(z))\exp(-\int_{0}^{\delta}h(s)\dd s),
\end{equation}
is included in our setup, and lies in $\mathcal{B}^{\tilde{\psi}_{0,\beta}}_k(H_0\times\mathbb{R})$ for all $k\ge 0$ if $\beta>0$ is chosen large enough:
first, note that $f$ depends on $h$ only via the bounded linear functional $\ell\colon H_0\to\mathbb{R}$, $\ell(h):=\int_{0}^{\delta}h(s)\dd s$.
It follows that
\begin{equation}
	\lvert f(h,v,z)\rvert
	\le
	\exp(2K)\exp(C\lVert h\rVert_{H_0}).
\end{equation}
Choosing $\beta>C$, the claim is proved, as it is clear that we can approximate $f$ by functions of the form $(h,z)\mapsto\exp(-\Phi(z))\varphi(\ell(h))$ with $\varphi\in\mathrm{C}_b^{\infty}(\mathbb{R})$ in the norm of $\mathcal{B}^{\tilde{\psi}_{0,\beta}}_k(H_0\times\mathbb{R})$.

While standard payoffs, such as caplets and swaptions, do not satisfy the smoothness assumptions required in our results, we can at least prove that they are contained in a space on which convergence~-- albeit without rates~-- is ensured.
A similar argument as for the bond price can be used to prove that the modified payoffs of caplets, 
\begin{equation}
	f(h,v,z)=\exp(-\Phi(z))(L_{\delta}(h)-K)_+, 
\end{equation}
where $L_{\delta}(h):=\frac{1}{\delta}\left( \exp\left( \int_{0}^{\delta}h(\tau)\dd\tau \right) - 1 \right)$ is the LIBOR rate, and payer swaptions, 
\begin{align}
	& f(h,v,z) = \nonumber \\
& =\exp(-\Phi(z))\left[ \sum_{i=1}^{I}\exp(-\int_{0}^{i\delta}h(\tau)\dd\tau)\left( \exp\left( \int_{(i-1)\delta}^{i\delta}h(\tau)\dd\tau \right)-(1+\delta K) \right) \right]_+, \nonumber
\end{align}
are contained in $\mathcal{B}^{\tilde{\psi}_{0,\beta}}(H_0\times\mathbb{R})$.
Here, however, taking the positive part makes these functions nonsmooth.
As the space $\mathcal{B}^{\tilde{\psi}_{0,\beta}}_{2(s+1)}(H_0\times\mathbb{R})$ of functions on which a rate of convergence is proved is dense in $\mathcal{B}^{\tilde{\psi}_{0,\beta}}(H_0\times\mathbb{R})$, we still obtain convergence.

\section{Numerics for the Heath-Jarrow-Morton equation}
We present the results of numerical computations for a Heath-Jarrow-Morton model. We do neither claim that the chosen HJM model is particularly well suited nor that the chosen calibration strategy is the best. We only want to demonstrate that a non-linear infinite-dimensional HJM model with stochastic volatility can be efficiently calibrated to market data with a satisfactory result.

First, a numerical calibration to caplet prices is performed, afterwards, a payer swaption is priced using the calibrated model.
In our numerics, space discretisation is performed using piecewise affine and continuous functions, where the mesh is aligned with the time mesh.
Hence, the partial differential equation
\begin{equation}
	\frac{\partial}{\partial t}r_{0,0}(t,r_0)(x) = \frac{\partial}{\partial x}r_{0,0}(t,r_0)(x),
	\quad
	r_{0,0}(0,r_0)(x) = r_0(x),
\end{equation}
is solved exactly by shifting $r_0$.

\subsection{Calibration}
We demonstrate the efficiency of the presented method by performing the calibration of a parametrised, time-homogeneous Heath-Jarrow-Morton model to the caplet volatility surface provided in \cite{Kluge2005}.
Note that the bond prices given there are automatically reproduced in our model by choosing them as the initial value.

We set $d=3$, and specify $\sigma_j(h,v)=g_j(h,v)\lambda_j$.
Here, $\lambda_j$ is assumed to be of the exponential-polynomial type \cite{Filipovic2001}, $\lambda_j(x)=\sum_{i=0}^{i_0}\alpha_{j,i}x^i\exp(-\beta x)$.
It is easy to see that under such assumptions, the regularity required in \cite[Section~5.2]{Filipovic2001} is satisfied.
In our experiments, we choose $i_0=2$.

There are several economically sound possibilities for choosing $g_j$.
Guided by the Cox-Ingersoll-Ross model, one could choose $g_j(h,v)=\sqrt{\lvert v h(t_j)\rvert}$ with some $t_j\ge 0$, where the absolute values are necessary as we cannot guarantee positive interest rates by this approach.
This ansatz, however, is not contained in our general setup, as $g_j$ is not a smooth function of $h$.

\begin{figure}[htpb]
	\centering
	\includegraphics[width=13cm]{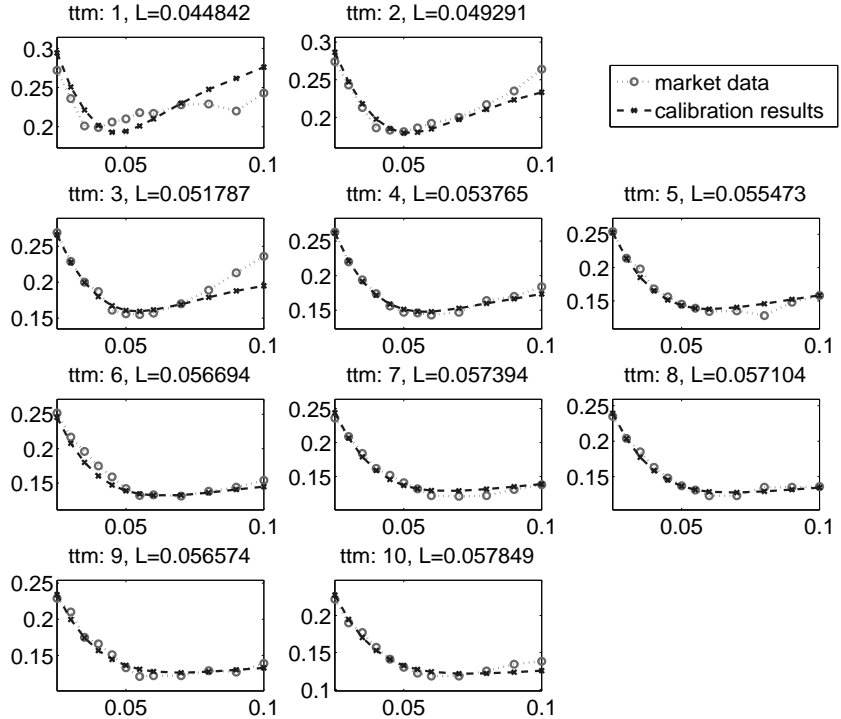}
	\caption{Calibration of the $\tanh$-type volatilities}
	\label{fig:calibration-tanh-sv}
\end{figure}
Instead, we assume $g_j(h,v)=\tanh(c_j\exp(v)\int_0^{t_j}h(s)\dd s)$.
This ensures that the volatilities are bounded and vanish if the benchmark yields $\int_{0}^{t_j}h(s)\dd s$ driving the equation go to zero.
We discretize the HJM-equation by the symmetrically weigthed sequential splitting scheme, as described in Section \ref{sec:swss}.
The calibration is performed by combining a custom-written genetic algorithm, searching for global minima, with the Levenberg-Marquardt implementation from \cite{Lourakis2004} to optimise locally.
The model caplet values are calculated numerically, using $12$ time steps per year and $2048$ quasi-Monte Carlo paths, based on the direction vectors for Sobol$'$ sequences of Joe and Kuo \cite{JoeKuo2008}.

All in all, $13$ parameters are used to match $120$ prices, and total calibration time is $14.5$ minutes running on 16 cores of a Primergy RX200 S6 spotting 4 Intel Xeon CPU X5650 processor, each of which provides 6 cores.  The calculation of $120$ option prices takes about $.5$ seconds and therefore merits to be called efficient.

We are able to match the market volatilites taken from \cite{Kluge2005} very well using the $\tanh$-type volatilies. 
Only the error in the earlier time slices is significant, see Figure~\ref{fig:calibration-tanh-sv}.
This is typical for models without jumps.
These are well known to misprice options close to maturity.
This behaviour can also be connected to the short end of interest rates depending more on announcements by central banks than random fluctuations.

With respect to the martingale property of traded assets, numerical calculations show that bond prices and LIBOR rates satisfy the expected value property to a very high precision already using $2048$ quasi-Monte Carlo paths.

\subsection{Pricing}
As an application, we price an at the money payer swaption with a time to maturity of $T=5$ years, where the underlying swap pays out quarter annually over three years, i.e., at the times $T_i=T+i\delta$ for $i=1,\dots,12$ and $\delta=.25$.
A reference computation with $16384$ paths and $120$ time steps per year yields the value $0.0281579$.
Using $2048$ paths and $12$ time steps per year, as in the calibration, we obtain $0.028074$.
The relative error is thus approximately $.003$.
As the calculation of the coarser approximation takes $.25$ seconds, we have established the efficiency of the suggested method.

\section{Conclusions}
We introduce an analytic setup for the analysis of weak approximation methods for stochastic partial differential equations.
The Heath-Jarrow-Morton equation of interest theory is shown to be included in the class where this approach is applicable.
Moreover, the set of admissible test functions contains important payoffs such as caplets and swaptions. We argue that higher-order weak approximation
schemes can be used together with QMC algorithms to obtain an efficient pricing method, which is even superior to multi-level MC. 
The efficiency of our numerical method is proved by the calibration of the model to given caplet data.

\appendix
\section{Functional analytic results}
\begin{proposition}
	\label{prop:simultaneousOB}
	Let $(X,\langle\cdot,\cdot\rangle_{X})$, $(Y,\langle\cdot,\cdot\rangle_Y)$ be separable Hilbert spaces with norms $\lVert\cdot\rVert_X$ and $\lVert\cdot\rVert_Y$ such that $Y$ is compactly and densely embedded into $X$.
	Then, there exists an orthonormal basis $(e_n)_{n\in\mathbb{N}}\subset Y$ of $X$ that is simultaneously orthogonal in $Y$.
	Furthermore, $\lim_{n\to\infty}\lVert e_n\rVert_{Y}^{-1}=0$.
\end{proposition}
\begin{proof}
	By the Riesz representation theorem, there exists a bounded operator $\kappa\colon X\to Y$ such that
	\begin{equation}
		\langle \kappa x,y\rangle_{Y} = \langle x,y\rangle_X
		\quad\text{for all $x\in X$ and $y\in Y$}.
	\end{equation}
	With $\iota\colon Y\to X$ the compact embedding, we set $K:=\iota\kappa$.
	$K$ is clearly compact and also symmetric, as
	\begin{equation}
		\langle Kx_1,x_2\rangle_X
		=
		\langle \kappa x_1,\kappa x_2\rangle_Y
		=
		\langle x_1,Kx_2\rangle_X.
	\end{equation}
	Thus, there exists an orthonormal basis $(e_n)_{n\in\mathbb{N}}$ of $X$ and a sequence $(\lambda_n)_{n\in\mathbb{N}}\subset\mathbb{R}$ decreasing monotonically to zero such that $Ke_n=\lambda_n e_n$, and we see that $(e_n)_{n\in\mathbb{N}}\subset Y$.
	We obtain
	\begin{equation}
		\langle e_n,e_m\rangle_{Y}
		=
		\lambda_n^{-1}\langle Ke_n,e_m\rangle_Y
		=
		\lambda_n^{-1}\langle e_n,e_m\rangle_X
		=
		\lambda_n^{-1}\delta_{n,m}
		\quad\text{for $n$, $m\in\mathbb{N}$},
	\end{equation}
	whence $(e_n)_{n\in\mathbb{N}}$ is orthogonal in $Y$ and $\lVert e_n\rVert_Y = \lambda_n^{-1/2}$, and the claim is proved.
\end{proof}
\begin{corollary}
	\label{cor:simultaneousOBapproximation}
	Under the assumptions of Proposition~\ref{prop:simultaneousOB}, let $\pi_N$ denote the $X$-orthogonal projection onto $X_N:=\lspan\left\{ e_n\colon n\in\mathbb{N} \right\}$.
	Then, 
	\begin{equation}
		\lim_{n\to\infty}\sup_{\substack{y\in Y\\\lVert y\rVert_Y\le 1}}\lVert y-\pi_N y\rVert_X = 0.
	\end{equation}
\end{corollary}
\begin{proof}
	By Parseval's identity,
	\begin{alignat}{2}
		\lVert y-\pi_N y\rVert_X^2
		&=
		\sum_{n=N+1}^{\infty}\langle y,e_n\rangle_X^2
		=
		\sum_{n=N+1}^{\infty}\langle y,Ke_n\rangle_Y^2 \notag \\
		&\le
		\sup_{n>N}\lambda_n\sum_{n=N+1}^{\infty}\langle y,\lambda_n^{1/2}e_n\rangle_Y^2
		\le
		\lambda_{N+1}\lVert y\rVert_Y,
	\end{alignat}
	where we apply that $\lVert e_n\rVert_Y=\lambda_n^{-1/2}$ and that $(e_n)_{n\in\mathbb{N}}$ is orthogonal in $Y$.
	As $(\lambda_n)_{n\in\mathbb{N}}$ decreases to zero, the claim follows.
\end{proof}

\subsection*{Acknowledgements}
	The numerical calculations were performed on the computing facilities of the Departement Mathematik of ETH Z\"urich. 
	Parts of the computer implementation were written by Dejan Velu\v{s}\v{c}ek, whom the authors thank for his support.
	Financial support from the ETH Foundation is gratefully acknowledged.

\bibliographystyle{amsplain}
\bibliography{lit,mylit,mythesis}
\end{document}